\documentclass[reqno,11pt]{amsart}

\usepackage{amssymb}
\usepackage{mathrsfs}
\usepackage{cite}
\usepackage{amsfonts}

\def\module#1{\mathrm{mod}\text{-}#1}

\def\pv#1{\ensuremath{{\bf#1}}}
\def\Ann#1{\mathrm{Ann}_S(#1)}

\def\p{\varphi}

\def\J{\mathrel{{\mathscr J}}} 

\def\R{\mathrel{{\mathscr R}}} 
\def\L{\mathrel{{\mathscr L}}} 

\def\e<{\leq _{E}}

\def\malce{\mathbin{\hbox{$\bigcirc$\rlap{\kern-8.3pt\raise0,50pt\hbox{$\mathtt{m}$}}}}}

\def\1sk{^{(1)}}

\def\to{\rightarrow}

\def\Thmname{Theorem}
\def\Propname{Proposition}
\def\Lemmaname{Lemma}
\def\Definitionname{Definition}
\newtheorem{Thm}{\Thmname}

\newtheorem{Fact}[Thm]{Fact}
\newtheorem{Lemma}[Thm]{\Lemmaname}
\newtheorem{Def}[Thm]{\Definitionname}

\newtheorem{Cor}[Thm]{Corollary}

\numberwithin{equation}{section}

\title{On the irreducible representations of a finite semigroup}

\author[O.~Ganyushkin]{Olexandr Ganyushkin}
\address{Department of Mechanics and Mathematics\\ 
Kyiv Taras Shev\-chen\-ko University\\ 
64, Volodymyrska st.\\
 01033, Kyiv, UKRAINE,}
\thanks{The first author was supported in part by STINT}
\email{ganiyshk\symbol{64}univ.kiev.ua}

\author[V.~Mazorchuk]{Volodymyr Mazorchuk}
\address{Department of Mathematics\\ 
Uppsala University\\ 
SE 471 06\\
Uppsala, SWEDEN\\
and Department of Mathematics\\ 
University of Glasgow\\ 
University Gardens\\
Glasgow G12 8QW, UK}
\thanks{The second author was supported in part by the Swedish Research Council}
\email{mazor\symbol{64}math.uu.se, v.mazorchuk@maths.gla.ac.uk}

\author[B.~Steinberg]{Benjamin Steinberg}
\address{Carleton University \\
1125 Colonel By Drive\\
Ottawa, Ontario  K1S 5B6 \\
CANADA}
\thanks{The third author was supported in part by NSERC}
\email{bsteinbg@math.carleton.ca}
\date{December 18, 2007}


\begin{document}
\begin{abstract}
Work of Clifford, Munn and Ponizovski{\u\i} parameterized the irreducible
representations of a finite semigroup in terms of the irreducible
representations of its maximal subgroups.  Explicit constructions of
the irreducible representations were later obtained independently by
Rhodes and Zalcstein and by Lallement and Petrich.  All of these
approaches make use of Rees's
theorem characterizing $0$-simple semigroups up to isomorphism.  Here
we provide a short modern proof of the Clifford-Munn-Ponizovski{\u\i}
result based on a lemma of
J.~A.~Green, which allows us to circumvent the theory of $0$-simple
semigroups.  A novelty of this approach is that it works over any base
ring.
\end{abstract}
\maketitle

\section{Introduction and preliminaries}
Work of Clifford~\cite{Clifford1,Clifford2}, Munn~\cite{Munn1,Munn2} and Ponizovski{\u\i}~\cite{Poni} parameterized the irreducible
representations of a finite semigroup in terms of the irreducible
representations of its maximal subgroups. (See~\cite[Chapter 5]{CP}
for a full  account of this work.)  Explicit constructions of
the irreducible representations were later obtained independently by
Rhodes and Zalcstein~\cite{RhodesZalc} and by Lallement and
Petrich~\cite{LallePet} in terms of the Sch\"utzenberger representation by
monomial matrices~\cite{Schutzmonomial}.  All of these approaches
make use of Rees's theorem~\cite{Rees} characterizing $0$-simple semigroups up to
isomorphism, thereby rendering the results somewhat inaccessible to
the non-specialist in semigroup theory.  As a consequence, it seems that when
researchers from other areas need to use semigroup representation theory,
they are forced to reinvent parts of the theory for
themselves, e.g.~\cite{Brown1,Brown2}.  This paper,
like~\cite{mobius1,mobius2,Putcharep3}, aims to reconcile semigroup
representation theory with representation theory at large.

The goal of this note is to give a self-contained accounting of the
theory of simple modules over the semigroup algebra of a finite semigroup
using only the tools of associative algebras.  This should make
the results accessible to the general mathematician for the first
time.  Our key tool is a lemma of J.~A. Green~\cite{Greenpoly}.  An advantage of this approach is
that it avoids Wedderburn theory and hence works over an arbitrary commutative ring with unit.

We collect here some basic definitions and facts concerning finite
semigroups that can be found in any of~\cite{Arbib,CP,qtheor}.  Let
$S$ be a (fixed) finite semigroup.  If $e$ is an idempotent, then
$eSe$ is a monoid with
identity $e$; its group of units $G_e$ is called the \emph{maximal
  subgroup} of $S$ at $e$.  Two idempotents $e,f$ are said to be
\emph{isomorphic} if there exist $x\in eSf$ and
$x'\in fSe$ such that $xx'=e, x'x=f$.  In this case one can show that
$eSe$ is isomorphic to $fSf$ as
monoids and hence $G_e\cong G_f$.

If $s\in S$, then $J(s) = S^1sS^1$ is the principal (two-sided) ideal
generated by $s$ (here $S^1$ means $S$ with an adjoined identity).
Following Green~\cite{Green}, two elements $s,t$ of a semigroup $S$
are said to be \emph{$\J$-equivalent} if $J(s)=J(t)$.  In this case
one writes $s\J t$.   In fact there is a preorder on $S$ given by
$s\leq_{\J} t$ if $J(s)\subseteq J(t)$.  This preorder induces an
ordering on $\J$-classes in the usual way.

\begin{Fact}
In a finite semigroup, idempotents $e,f$ are isomorphic if and only if $e\J f$, that is, $SeS=SfS$.
\end{Fact}

An element $s$ of a semigroup $S$ is said to be (von Neumann)
\emph{regular} if $s=sts$ for some $t\in S$.  Each idempotent is of
course regular.

\begin{Fact}\label{regularJclass}
Let $S$ be a finite semigroup and $J$ a $\J$-class of $S$.  Then the
following are equivalent:
\begin{enumerate}
\item $J$ contains an idempotent;
\item $J$ contains a regular element;
\item all elements of $J$ are regular;
\item $J^2\cap J\neq \emptyset$.
\end{enumerate}
\end{Fact}

A $\J$-class satisfying the equivalent conditions in
Fact~\ref{regularJclass} is called a \emph{regular} $\J$-class.  The
poset of regular $\J$-classes is denoted $\mathscr U(S)$.  It was
shown by Putcha~\cite{Putcharep3} that over an algebraically closed 
field of characteristic zero the module category of a regular semigroup 
(one in which all elements are regular)  is a highest
weight category~\cite{quasihered} with weight poset $\mathscr U(S)$.  We
need one last fact about finite semigroups in order to state and prove
the Clifford-Munn-Ponizovski{\u\i} theorem.

\begin{Fact}\label{dropoutofJ}
Let $S$ be a finite semigroup and $J$ a regular $\J$-class. Let $e\in
J$ be an idempotent.  Then $eSe\cap J=G_e$.
\end{Fact}

Let $J$ be a $\J$-class of $S$.  Set $I_J = \{s\in S\mid J\nsubseteq
J(s)\}$; it is the ideal of all elements of $S$
that are not $\J$-above some (equals any) element of $J$.

\section{Characterization and Construction of Simple Modules}
Fix a finite semigroup $S$ and a commutative ring with unit $K$.  The
semigroup algebra
$KS$ need not be unital. If $A$ is a $K$-algebra,
not necessarily unital, then by a \emph{simple module} $M$, we mean a
(right) $A$-module $M$ such that $MA\neq 0$ and $M$
contains no proper non-zero submodules, or equivalently for all $0\neq
m\in M$, the cyclic module $mA=M$.   Of
course if $K$ is a field and $A$ is finite-dimensional, then every simple $A$-module is finite 
dimensional, being cyclic and hence a quotient of the regular module
$A$.  The category of (right) $A$-modules will be denoted $\module A$.
We adopt the convention that if $A$ is unital, then by $\module A$ we
mean the category of unital $A$-modules.  The reader should verify
that all functors considered in this paper respect this convention.  

If $M$ is a $KS$-module, then $\Ann M=\{s\in S\mid Ms=0\}$.
Clearly $\Ann M$ is an ideal of $S$.
The following definition, due to Munn~\cite{Munn1}, is crucial to semigroup representation
theory.

\begin{Def}[Apex]
A regular $\J$-class $J$ is said to be the apex of a $KS$-module $M$
if $\Ann M=I_J$.
\end{Def}
It is easy to see that $M$ has apex $J$ if and only if $J$ is the
unique $\leq_{\J}$-minimal $\J$-class that does not annihilate $M$.

Fix an idempotent transversal $E=\{e_J\mid J\in \mathscr U(S)\}$ to
the set $\mathscr U(S)$ of regular $\J$-classes  and set $G_J=
G_{e_J}$.  Let $A_J=KS/KI_J$.  Notice that the category of
$KS$-modules with apex $J$ can be identified with the full subcategory
of $\module {A_J}$ whose objects are modules $M$ such that
$Me_J\neq 0$.

Our first goal is to show that every simple module has an apex.  This
result is due independently to Munn and
Ponizovski{\u\i}~\cite{Munn1,Munn2,Poni}.

\begin{Thm}\label{haveapex}
Let $M$ be a simple $KS$-module.  Then $M$ has an apex.
\end{Thm}
\begin{proof}
Since $MKS\neq 0$, there is a $\leq_{\J}$-minimal $\J$-class $J$ so
that $J\nsubseteq \Ann M$.  Let $I = S^1JS^1$; of course, $I$ is an
ideal of $S$. Since $I\setminus J$ annihilates $M$
by minimality of $J$, it follows $0\neq MKJ=MKI$. From the fact that $I$ is an
ideal of $S$, we may deduce that $MKI$ is a $KS$-submodule and so  by simplicity
\begin{equation}\label{apex1}
M= MKI = MKJ.
\end{equation}
Therefore, since $JI_J\subseteq I\setminus J\subseteq \Ann M$, it follows
from \eqref{apex1} that $I_J= \Ann M$.  Now if $J$ is not regular,
then Fact~\ref{regularJclass} implies $J^2\subseteq I\setminus J$ and hence $J$
annihilates $M$ by \eqref{apex1}, a contradiction.  Thus $J$ is
regular and is an apex for $M$.
\end{proof}

Now we wish to establish a bijection between simple $KS$-modules with apex
$J$ and simple $KG_J$-modules.  This relies on a well-known result of
Green~\cite{Greenpoly}.   Let $A$ be an algebra and $e$ an idempotent
of $A$. Then $eA$ is an $eAe$-$A$-bimodule and $Ae$ is an $A$-$eAe$-bimodule.  Hence we have a
restriction functor $\mathrm{Res}:\module A\to \module {eAe}$ and induction/coinduction
functors $\mathrm{Ind},\mathrm{Coind}:\module {eAe}\to \module A$ given by
\begin{equation*}
\mathrm{Ind}(M) = M\otimes _{eAe}eA,\ \mathrm{Res}(M) = Me\ \text{and}\
\mathrm{Coind}(M) = \mathrm{Hom}_{eAe}(Ae,M).
\end{equation*}
Moreover,  $\mathrm{Ind}$ is right exact, $\mathrm{Res}$ is exact,  $\mathrm{Coind}$ is left exact and
$\mathrm{Ind}$ and $\mathrm{Coind}$ are the left and right adjoints of $\mathrm{Res}$, respectively.  This follows from observing that
$\mathrm{Hom}_A(eA,M) = Me = M\otimes_A Ae$ and the
usual adjunction between hom and the tensor product. Moreover, it is
well known that unit of the first adjunction gives a natural isomorphism
$M\cong \mathrm{Ind}(M)e$ while the counit of the second gives a
natural isomorphism $\mathrm{Coind}(M)e\cong M$.  

There also two important functors $N,L:\module A\to \module A$ given by
$N(M) = \{m\in M\mid mAe=0\}$ and $L(M) = MeA$.  It is easily verified that $N(M)$ is
the largest $A$-submodule of $M$ that is annihilated by $e$, while $L(M)$ is the smallest $A$-submodule of $M$ with $L(M)e=Me$.   Our
next result can be found in~\cite[6.2]{Greenpoly}, but we reproduce
the proof here for convenience of the reader.

\begin{Lemma}[Green]\label{greenslemma}
Let $A$ be an algebra and $e$ an idempotent of $A$.
\begin{enumerate}
\item If $M$ is a simple $A$-module, then $Me=0$ or $Me$ is a simple
  $eAe$-module.
\item  If $V$ is a simple $eAe$-module, then
  $\mathrm{Ind}(V)$ has unique maximal
  submodule  $N(\mathrm{Ind}(V))$ and
  $\mathrm{Ind}(V)/N(\mathrm{Ind}(V))$ is the unique simple
$A$-module $M$ with $Me\cong V$.
\item  If $V$ is a simple $eAe$-module, then $\mathrm{Coind}(V)$ has a unique minimal $A$-submodule $L(\mathrm{Coind}(V))$ and $L(\mathrm{Coind}(V))$ is the unique simple $A$-module $M$ with $Me\cong V$.
\end{enumerate}
Consequently, restriction yields a bijection between simple
  $A$-modules that are not annihilated by $e$ and simple
  $eAe$-modules.
\end{Lemma}
\begin{proof}
To prove (1), assume $Me\neq 0$.  Let $m\in
Me$ be non-zero.  Then $meA = mA=M$, so $meAe=Me$.  Thus $Me$ is
simple.

Now we turn to (2).  Let $V$ be a simple $eAe$-module and suppose
$w\in \mathrm{Ind}(V)$ with $w\notin N(\mathrm{Ind}(V))$.  Then $0\neq wAe\subseteq \mathrm{Ind}(V)e$.
But $\mathrm{Ind}(V)e\cong V$ is a simple $eAe$-module, so \[(wAeAe)A = \mathrm{Ind}(V)eA =
(V\otimes_{eAe}eA)eA = (V\otimes _{eAe}e)A = \mathrm{Ind}(V)\] and thus $N(\mathrm{Ind}(V))$ is
the unique maximal $A$-submodule of $\mathrm{Ind}(V)$.  In particular,
$\mathrm{Ind}(V)/N(\mathrm{Ind}(V))$ is a simple $A$-module.  Since restriction
is exact and $N(\mathrm{Ind}(V))e=0$ by construction, it follows
\[[\mathrm{Ind}(V)/N(\mathrm{Ind}(V))]e\cong \mathrm{Ind}(V)e/N(\mathrm{Ind}(V))e\cong \mathrm{Ind}(V)e\cong V.\]

It remains to prove uniqueness.  Suppose $M$ is a simple
$A$-module such that $Me\cong V$.  Then, using the adjunction between
induction and restriction, we have $\mathrm{Hom}_{eAe}(V,Me)\cong
\mathrm{Hom}_A(V\otimes_{eAe}eA, M)$.  Hence the isomorphism
$V\rightarrow Me$ corresponds to a non-zero homomorphism $\p:\mathrm{Ind}(V)\to M$,
which is necessarily onto as $M$ is simple.  But $N(\mathrm{Ind}(V))$ is the unique
maximal submodule of $\mathrm{Ind}(V)$, so $\ker \p=N(\mathrm{Ind}(V))$ and hence $M\cong
\mathrm{Ind}(V)/N(\mathrm{Ind}(V))$, as required.

Finally, to prove (3) first observe that if $M$ is any non-zero
$A$-submodule of $\mathrm{Coind}(V)$, then $Me\neq 0$.  Indeed,
suppose $Me=0$ and let $\p\in M$. Then, for any $x$ in $Ae$, we have
$\p(x) = (\p xe)(e)=0$ since $\p xe\in Me=0$.  It follows that $M=0$.
Since $\mathrm{Coind}(V)e\cong V$ is a simple $eAe$-module, it now
follows that if $M$ is a non-zero $A$-submodule of
$\mathrm{Coind}(V)$, then $Me=\mathrm{Coind}(V)e$ and hence
\[L(\mathrm{Coind}(V)) = \mathrm{Coind}(V)eA\subseteq MeA\subseteq
M.\]  This establishes that $L(\mathrm{Coind}(V))$ is the unique
minimal $A$-submodule.  Since $L(\mathrm{Coind}(V))e =
\mathrm{Coind}(V)eAe = \mathrm{Coind}(V)e\cong V$, it just remains to
prove uniqueness.  Suppose $M$ is a simple $A$-module with $Me\cong
V$. Then the existence of a non-zero element of
$\mathrm{Hom}_{eAe}(Me,V)\cong \mathrm{Hom}_A(M,\mathrm{Coind}(V))$
implies that $M$ admits a non-zero homomorphism to
$\mathrm{Coind}(V)$.  Hence $M$ is isomorphic to a simple
$A$-submodule of $\mathrm{Coind}(V)$. But $L(\mathrm{Coind}(V))$ is
the unique simple submodule of $\mathrm{Coind}(V)$ and so $M\cong
L(\mathrm{Coind}(V))$, as required. 
\end{proof}

We may now complete the proof of the Clifford-Munn-Ponizovski{\u\i} theorem,
with an explicit construction of the simple modules equivalent to the one
found in~\cite{RhodesZalc,LallePet}.

\begin{Thm}[Clifford, Munn, Ponizovski{\u\i}]
Let $S$ be a finite semigroup, $K$ a commutative ring with unit
and $E=\{e_J\mid
J\in \mathscr U(S)\}$ an idempotent transversal to the set $\mathscr
U(S)$ of regular
$\J$-classes of $S$.  Let $G_J$ be the maximal subgroup $G_{e_J}$.
Define functors $\mathrm{Ind}_{G_J}^S,\mathrm{Coind}_{G_J}^S:\module KG_J\to \module KS$ by
\begin{align*}
\mathrm{Ind}_{G_J}^S(V)&= V\otimes_{KG_J}e_J(KS/KI_J)\\
\mathrm{Coind}_{G_J}^S(V)&= \mathrm{Hom}_{KG_J}((KS/KI_J)e_J,V).
\end{align*}
Then:
\begin{enumerate}
\item If $M$ is a simple $KS$-module with apex $J$, then $Me_J$ is a
  simple $KG_J$-module;
\item If $V$ is a simple $KG_J$-module, then \[N = \{w\in
  \mathrm{Ind}_{G_J}^S(V)\mid wKSe_J=0\},\] is the unique maximal
  $KS$-submodule of $\mathrm{Ind}_{G_J}^S(V)$ and
  $\mathrm{Ind}_{G_J}^S(V)/N$ is the unique simple $KS$-module $M$
with apex $J$ such that $Me_J=V$;
\item If $V$ is a simple $KG_J$-module, then
  $\mathrm{Coind}_{G_J}^S(V)eA$ is the unique minimal $A$-submodule of
  $\mathrm{Coind}_{G_J}^S(V)$ and moreover is the unique simple
  $KS$-module $M$ 
with apex $J$ such that $Me_J=V$.
\end{enumerate}
Consequently, if $K$ is a field there is a bijection between
irreducible representations
of $S$ and irreducible representations of the maximal subgroups $G_J$ of $S$,
$J\in \mathscr U(S)$.
\end{Thm}
\begin{proof}
Theorem~\ref{haveapex} implies that every simple $KS$-module $M$ has an
apex. Again setting $A_J=KS/I_J$ for a regular $\J$-class $J$, we know
that simple $KS$-modules with apex $J$ are in bijection with simple
$A_J$-modules $M$ such that $Me_J\neq 0$. It follows directly from
Fact~\ref{dropoutofJ} that \[e_JA_Je_J=Ke_JSe_J/Ke_JI_Je_J=KG_J.\]
Lemma~\ref{greenslemma} then yields that simple $A_J$-modules not
annihilated by $e_J$, that is simple $KS$-modules with apex $J$, are
in bijection with simple $KG_J$-modules in the prescribed manner.
\end{proof}

Let us make a remark to relate the above construction of the simple
modules to the ones found in~\cite{RhodesZalc,LallePet}.   All the
facts about finite semigroups used in this discussion can be found in the
appendix of~\cite{qtheor} or in~\cite{Arbib}. According to
Green~\cite{Green}, two elements $s,t$ of a semigroup are said to be
\emph{$\mathscr R$-equivalent} if they generate the same principal right
ideal.  Dually $s,t$ are said to be \emph{$\L$-equivalent} if they generate
the same principal left ideal.  If $S$ is a finite semigroup, then it
is well known (retaining our previous notation) that $e_JS\cap J$ is the $\R$-class $R_{e_J}$ of $e_J$ and $Se_J\cap J$ is the $\L$-class
$L_{e_J}$ of $e_J$.  Furthermore, left multiplication yields a free action of
$G_J$ on the left of $R_{e_J}$ by automorphisms of the action of $S$
on the right of $R_{e_J}$ by partial transformations (induced by right
multiplication).  Moreover, the $G_J$-orbits on $R_{e_J}$ are in bijection
with the set of $\L$-classes of $J$.  Let $T$ be a transversal to the
$G_J$-orbits.  Now $e_JKS/I_J$ can be identified as a vector space
with $KR_{e_J}$ and the right $KS$-module structure is the
linearization of the right action of $S$ on $R_{e_J}$ described
above.  Moreover, under this identification, the left $KG_J$-module
structure on $KR_{e_J}$ is induced by the free left action of $G_J$ on
$R_{e_J}$ and so $KR_{e_J}$ is a free left $KG$-module with basis
$T$. In particular, the functor $\mathrm{Ind}_{G_J}^S$ is exact.
It is straightforward to show using~\cite[Theorem~10.4.1]{GM} that under the
usual identification of maximal subgroups inside $J$, the 
$KG_J$-$KS$-bimodule $KR_{e_J}$ does not depend (up to isomorphism) on the
choice of $e_J\in J$. 

From the above it follows that  
$\mathrm{End}_{KG_J}(KR_{e_J})\cong M_n(KG_J)$ where $n$
is the number of $\L$-classes in $J$ and so there results a representation
$\rho_J:S\to M_n(KG_J)$, which is easily checked to be the classical right
Sch\"utzenberger representation by row monomial
matrices~\cite{CP,Schutzmonomial} since if $s\in S$ and $t\in T$, then
either $ts=0$ or $ts = gt'$ for unique elements $t'\in T$ and $g\in
G_J$.

Now let $V$ be a simple $KG_J$-module affording the
irreducible representation $\p:G_J\to GL_r(K)$. Then the matrix
representation afforded by the module \mbox{$V\otimes_{KG_J}KR_{e_J}$} is
the tensor product of $\p$ with $\rho_J$.  Now an element of $Se_J$
which does not belong to $J$ automatically annihilates
$V\otimes_{KG_J}KR_{e_J}$, so the unique maximal submodule consists of
those vectors annihilated by $Se_J\cap J = L_{e_J}$, the $\L$-class of
$e_J$.  If one chooses Rees matrix coordinates for $J$~\cite{CP,qtheor}, then
it is not hard to show that the vectors annihilated by the $\L$-class
of $e_J$ are those belonging to the null space of the image of the
sandwich matrix under $\p$.  Hence the construction of the simple
modules we have provided corresponds exactly to the construction found
in~\cite{RhodesZalc,LallePet}, but our proof avoids Rees matrix
semigroups and Munn algebras.

The coinduced module also has a natural semigroup theoretic
interpretation.  Indeed,   \mbox{$\mathrm{Hom}_{KG_J}((KS/KI_J)e_J,V)\cong
\mathrm{Hom}_{G_J}(L_{e_J},V)$} where we view $L_{e_J}$ and $V$ as right
$G_{J}$-sets.  The semigroup
$S$ acts on the left of $L_{e_J}$ by the left
Sch\"utzenberger representation and this induces the $KS$-module
structure on $\mathrm{Hom}_{G_J}(L_{e_J},V)$. Since $L_{e_J}$ is a
free right $G_{e_J}$-set and the orbits are in bijection with the set
of $\R$-classes in $J$, elements of $\mathrm{Hom}_{G_J}(L_{e_J},V)$
are in bijection with elements of $V^m$ where $m$ is the number of
$\R$-classes of $J$.  The space $V^m$ (viewed as row vectors) is
naturally a right $KS$-module 
via the left Sch\"utzenberger representation $\lambda_J:S\to M_m(KG_J)$
and the module structure agrees with the original one.
If one chooses Rees matrix coordinates for $J$~\cite{CP,qtheor}, then
the structure matrix $C$ takes $V^n$ to $V^m$ where $n$ is the number
of $\L$-classes of $S$. One can verify that the image of $C$ is the
unique minimal
$KS$-submodule of $V^m$. (The fact that it
is a submodule is a consequence of the so-called linked equations~\cite{Arbib,qtheor}.) This yields the
other construction of the irreducible representations
found in~\cite{RhodesZalc}.  Since $L_{e_J}$ is a free right
$G_J$-set, it follows that $\mathrm{Coind}_{G_J}^S$ is exact.  It
should be mentioned that all coinduced 
constructions can be obtained from induced constructions for the
opposite semigroup via duality.

Putcha has used both
the induced and coinduced modules, which he calls the left and right
induced modules, in his work on representation
theory~\cite{Putcharep3,Putcharep5}.

As an application, we provide the description of the irreducible
representations of an idempotent semigroup that was rediscovered by
Brown~\cite{Brown1,Brown2} and put to good effect in the study of random
walks.  First we establish a well-known lemma.

\begin{Lemma}\label{bandstable}
Let $S$ be a semigroup of idempotents and let $J$ be a $\J$-class of
$S$.  Then the complement of $I_J$ is a subsemigroup of $S$.
\end{Lemma}
\begin{proof}
First we show that $J$ is a subsemigroup.  Let $e,f\in J$.  Then we have
$e=ufv$ some $u,v\in S$ and so $efv = ufvfv=ufv=e$, establishing $ef\in
J$. Next suppose $J\subseteq SsS\cap Ss'S$. We need $J\subseteq Sss'S$.
Let $e\in J$.  Then $e=usv$ and 
$e=u's'v'$ with $u,v,u',v'\in S$.  Since $us(vus)v=e$ and
$u'(s'v'u')s'v'=e$, it follows $vus\J e\J s'v'u'$.  Since $J$ is a
subsemigroup, $vuss'v'u'\in J$ and 
hence $J\subseteq Sss'S$, as required.
\end{proof}

\begin{Cor}
Let $S$ be a finite semigroup all of whose elements are idempotents.
Then all the irreducible representations of $S$ over a field have
degree one and the unique irreducible representation $\p_J$ with apex
$J$ is given by
\begin{equation}\label{irreduciblesforbands}
\p_J(s) = \begin{cases} 0 & s\in I_J\\
1 & \text{otherwise.}
\end{cases}
\end{equation}
\end{Cor}
\begin{proof}
Let $J$ be a regular $\J$-class.  Lemma~\ref{bandstable} implies that
\eqref{irreduciblesforbands} is an irreducible representation with
apex $J$. It is afforded by $K$ with $S$-action \[ks 
= \begin{cases} 0 & s\in I_J\\ k & \text{otherwise.}\end{cases}\] 
Since $G_J$ is trivial, there is exactly one simple $KS$-module with
apex $J$, namely the quotient of 
$M=e_JKS/KI_J$ by its unique maximal submodule $N$.  Now
$R_{e_J}=e_JS\setminus I_J = e_JS\cap J$ is a basis for $M$.  As a
consequence of
Lemma~\ref{bandstable}, $R_{e_J}s \subseteq R_{e_J}$ for $s\in
S\setminus I_J$ and $R_{e_J}s\subseteq I_J$, otherwise.  Thus the
augmentation map $\varepsilon:M\to K$ 
  sending each element of $R_{e_J}$ to $1$ is a surjective morphism of
  $KS$-modules with kernel the unique maximal submodule $N$ of $M$, as
  $K$ is of course simple.  This completes the proof.
\end{proof}

The above argument applies {\it mutatis mutandis} to semigroups all of whose
subgroups are trivial and whose regular $\J$-classes are
subsemigroups.  This class of semigroups, known as $\pv {DA}$, was
introduced by Sch\"utzenberger in his study of unambiguous products of
regular languages~\cite{Schutznonambig}.

\bibliographystyle{abbrv}
\bibliography{standard2}

\end{document}